\pdfoutput=1
\documentclass[a4paper]{scrartcl}

\usepackage{amsthm}
\usepackage{graphicx}
\usepackage{overpic}
\usepackage{mathtools}
\usepackage{microtype}

\newtheorem{theorem}{Theorem}
\newtheorem{remark}{Remark}

\begin{document}

\title{A Parareal algorithm without Coarse Propagator?}

\author{Martin J. Gander\thanks{Université de Gen\`eve, martin.gander@unige.ch}
\and Mario Ohlberger\thanks{University of Münster, mario.ohlberger@uni-muenster.de}
\and Stephan Rave\thanks{University of Münster, stephan.rave@uni-muenster.de}}

\date{September 4, 2024}

\maketitle

\abstract{The Parareal algorithm was invented in 2001 in order to
  parallelize the solution of evolution problems in the time
  direction. It is based on parallel fine time propagators called F
  and sequential coarse time propagators called G, which alternatingly
  solve the evolution problem and iteratively converge to the fine
  solution. The coarse propagator G is a very important component of
  Parareal, as one sees in the convergence analyses. We present here
  for the first time a Parareal algorithm without coarse propagator,
  and explain why this can work very well for parabolic problems. We
  give a new convergence proof for coarse propagators approximating in
  space, in contrast to the more classical coarse propagators which
  are approximations in time, and our proof also applies in the
  absence of the coarse propagator. We illustrate our theoretical
  results with numerical experiments, and also explain why this
  approach can not work for hyperbolic problems.}

\section{Introduction}

The Parareal algorithm was introduced in \cite{lions2001resolution} as
a non-intrusive way to parallelize time stepping methods in the time
direction for solving partial differential equations. Its convergence
is now well understood: for linear problems of parabolic type,
Parareal converges superlinearly on bounded time intervals, and
satisfies a linear convergence bound on unbounded time intervals
\cite{gander2007analysis}, which means it is a contraction even if
computations are performed on arbitrary long time intervals. It was
shown in the same reference also that for hyperbolic problems, the
linear convergence estimate over long time does not predict
contraction, and the superlinear estimate only indicates contraction
when there are already too many iterations performed and one is
approaching the finite step convergence property of Parareal. A
non-linear convergence analysis for Parareal can be found in
\cite{gander2008nonlinear}, and also the performance of Parareal for
Hamiltonian problems is well understood, see
\cite{gander2014analysis}, where also a derivative variant of Parareal
was proposed.

An essential ingredient in the Parareal algorithm is the coarse
propagator, and it is its accuracy that has a decisive influence on
the convergence of Parareal, as one can see in the estimates from
\cite{gander2007analysis,gander2008nonlinear}. We consider here a
Parareal algorithm in the standard form for solving an evolution
problem $\partial_t u={\cal F}(u,t)$, and time partition
$0=T_0<T_1<\ldots<T_N=T$,
\begin{equation}\label{Parareal}
  U_{n+1}^{k+1}=F(U_n^k,T_n,T_{n+1})+G(U_n^{k+1},T_n,T_{n+1})-G(U_n^k,T_n,T_{n+1}),
\end{equation}
where the fine solver $F(U_n^k,T_n,T_{n+1})$ and the coarse solver
$G(U_n^k,T_n,T_{n+1})$ solve the underlying evolution problem on time
intervals $\Delta T:=T_{n+1}-T_n$,
\begin{equation}
  \partial_t u={\cal F}(u,t),\quad u(T_n)=U_n^k,
\end{equation}
with different accuracy. We are interested in understanding what
happens when we remove the coarse solver in the Parareal algorithm
\eqref{Parareal}, i.e.  we run instead the iteration
\begin{equation}
  U_{n+1}^{k+1}=F(U_n^k,T_n,T_{n+1}).
\end{equation}
Note that this is quite different from the Identity Parareal algorithm
\cite{takami2014identity}, where the coarse propagator was replaced by
the identity, a very coarse approximation, while here we remove the
coarse propagator altogether.

\section{A Numerical Experiment}

We start with a numerical experiment for the one dimensional heat
equation on the spatial domain $\Omega:=(0,1)$ and the time interval $(0,T]$,
\begin{equation}\label{HeatEquation1D}
  \partial_t u=\partial_{xx}u+f,\quad u(x,0)=u_0(x),\quad u(0,t)=u(1,t)=0.
\end{equation}
We show in Figure \ref{PararealFigWithoutCoarse}
\begin{figure}[t]
  \centering
  \mbox{\includegraphics[width=0.42\textwidth]{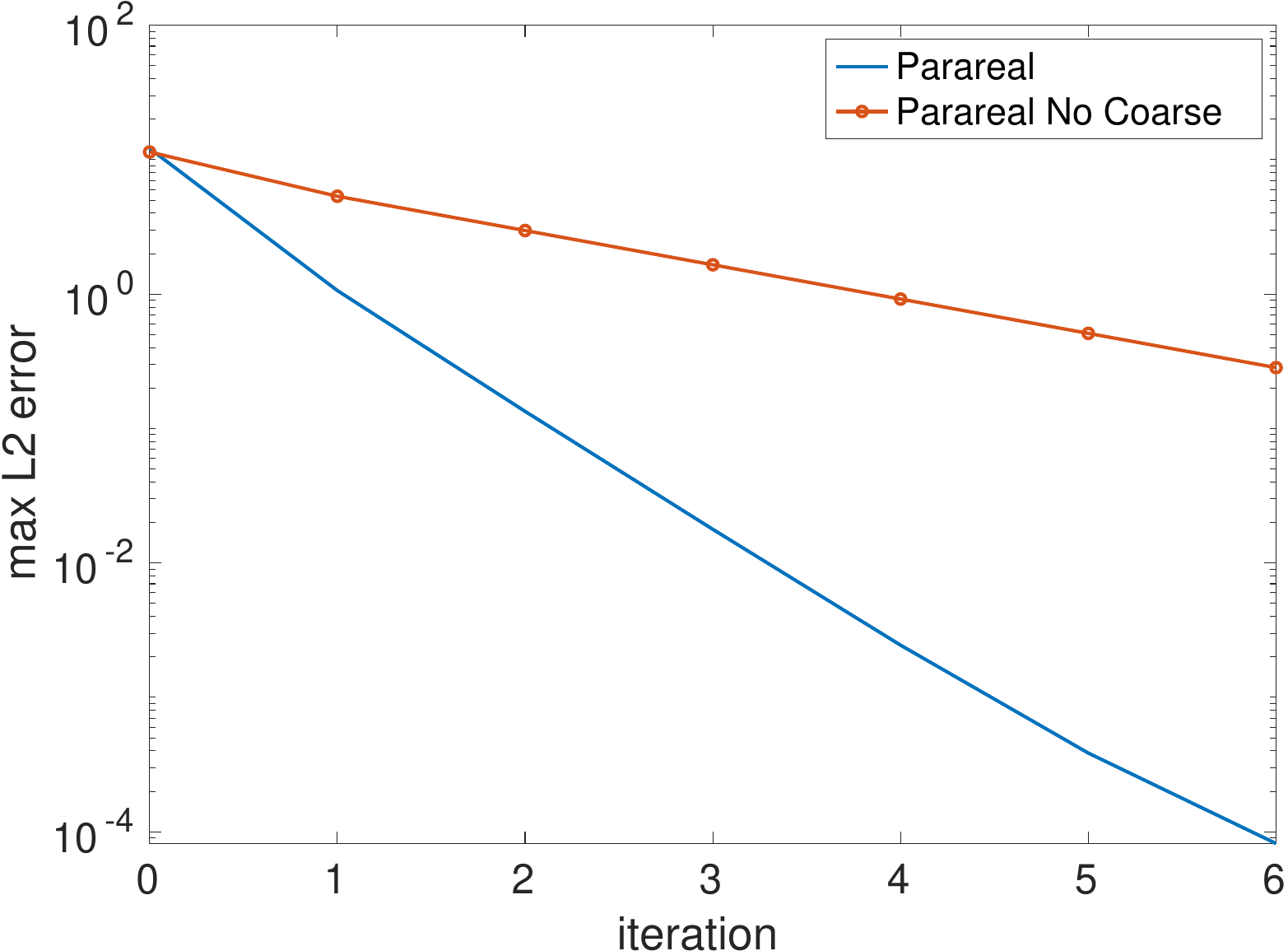}\qquad
    \includegraphics[width=0.42\textwidth]{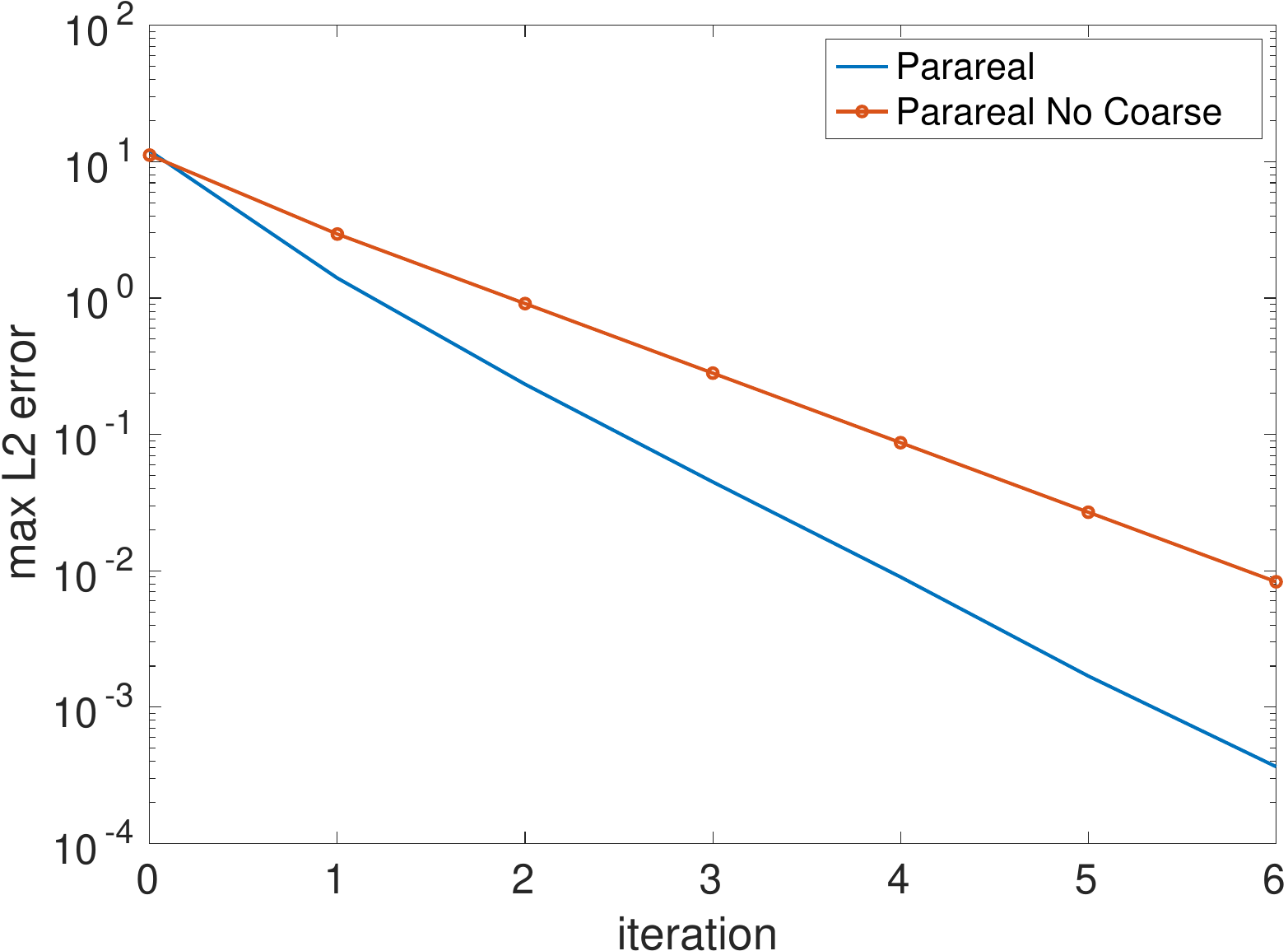}}
  \mbox{\includegraphics[width=0.42\textwidth]{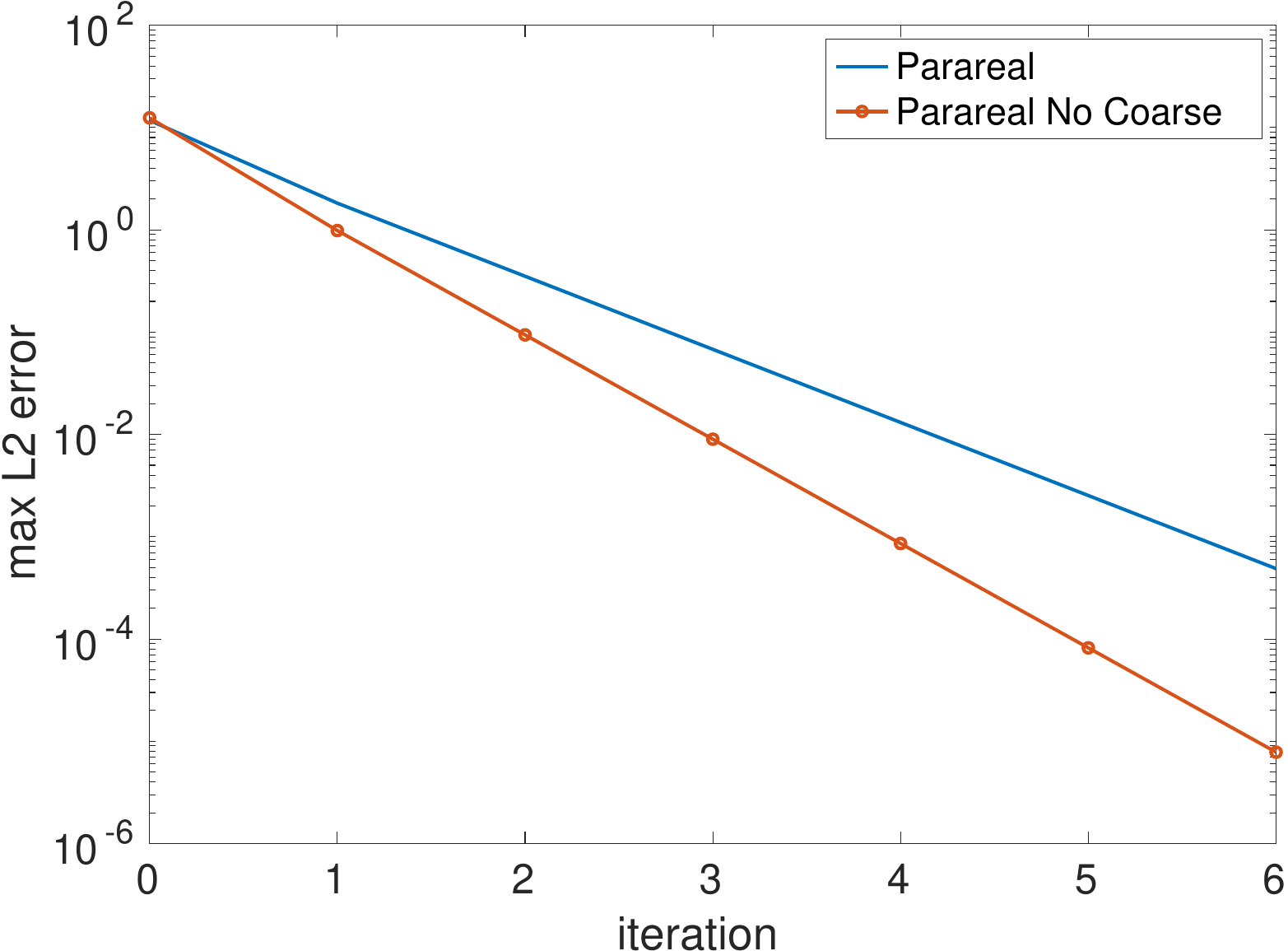}\qquad
    \includegraphics[width=0.42\textwidth]{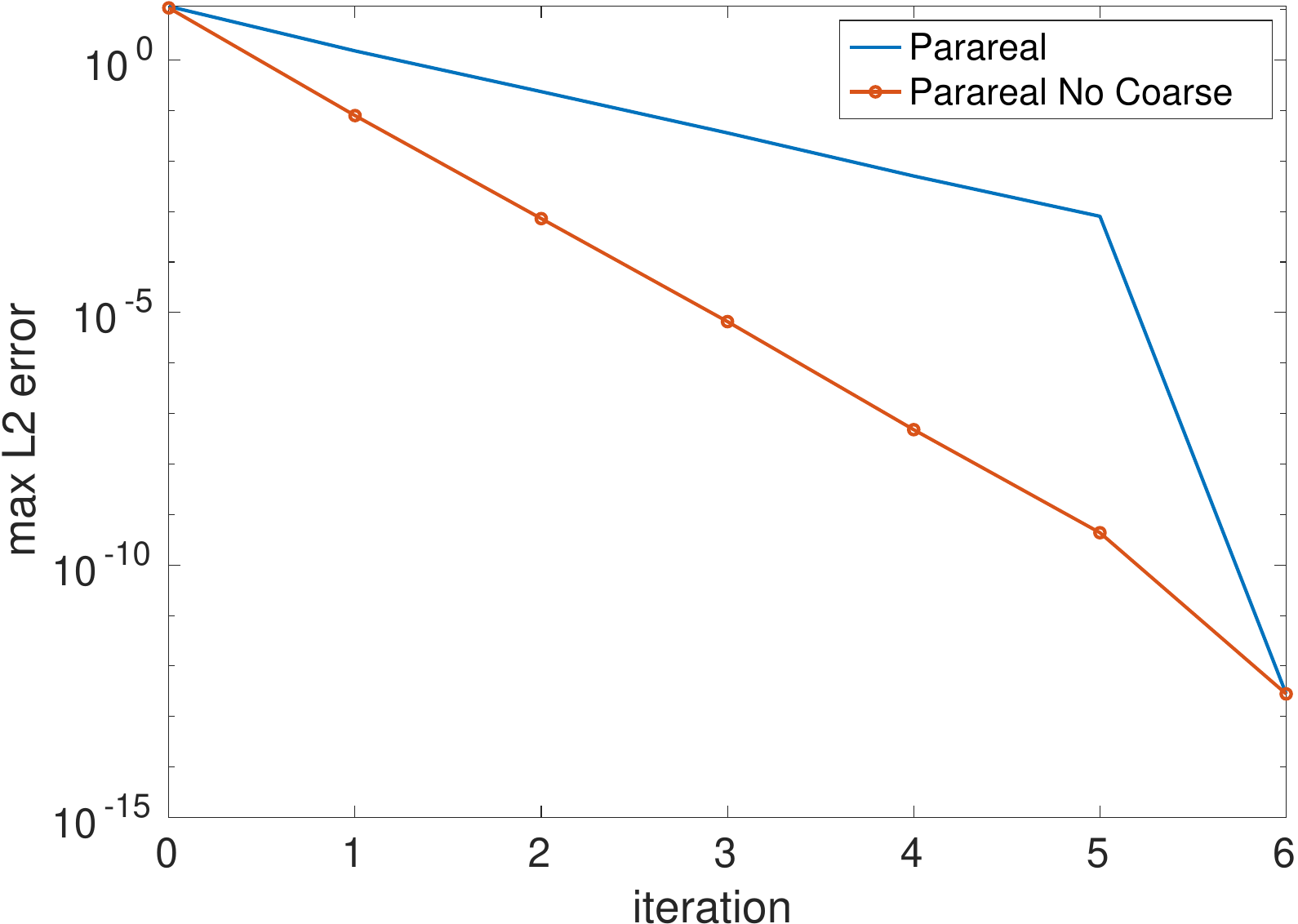}}
  \caption{Running a classical Parareal algorithm with and without
    coarse propagator.  From top left to bottom right: 48, 24, 12, and 6
    coarse time intervals with $T=3$.}
  \label{PararealFigWithoutCoarse}
\end{figure}
the results obtained with zero initial conditions, $u_0(x)=0$, and the
source term
\begin{equation}
  f(x,t)=10e^{-100(x-0.5)^2}\sum_{j=1}^4e^{-100(t-t_j)^2},\quad
  t_j=0.1,0.6,1.35,1.85,
\end{equation}
like a heating device which is turned on and of at specific time
instances. We discretized the heat equation \eqref{HeatEquation1D}
using a centered finite difference discretization in space with mesh
size $\Delta x=\frac{1}{128}$, and Backward Euler in time with time
step $\Delta t=\frac{1}{96}$, and run the Parareal algorithm up to
$T=3$ with various numbers $N$ of coarse time intervals, where the
coarse propagator $G$ just does one big Backward Euler step.  We see
two really interesting results: first, the Parareal algorithm without
coarse propagator is actually also converging when applied to the heat
equation, both methods are a contraction. And second, if there are not
many coarse time intervals, i.e. if the coarse time interval length
$\Delta T$ is becoming large, the Parareal algorithm without coarse
correction is converging even faster than with coarse correction!

In order to test this surprising property further, we repeat the
above experiment, but use now homogeneous Neumann boundary conditions
in \eqref{HeatEquation1D}, $\partial_xu(0,t)=\partial_xu(1,t)=0$,
instead of the Dirichlet conditions. We show the corresponding results
in Figure \ref{PararealFigWithoutCoarseNeumann}.
\begin{figure}[t]
  \centering
  \mbox{\includegraphics[width=0.42\textwidth,height=0.31\textwidth]{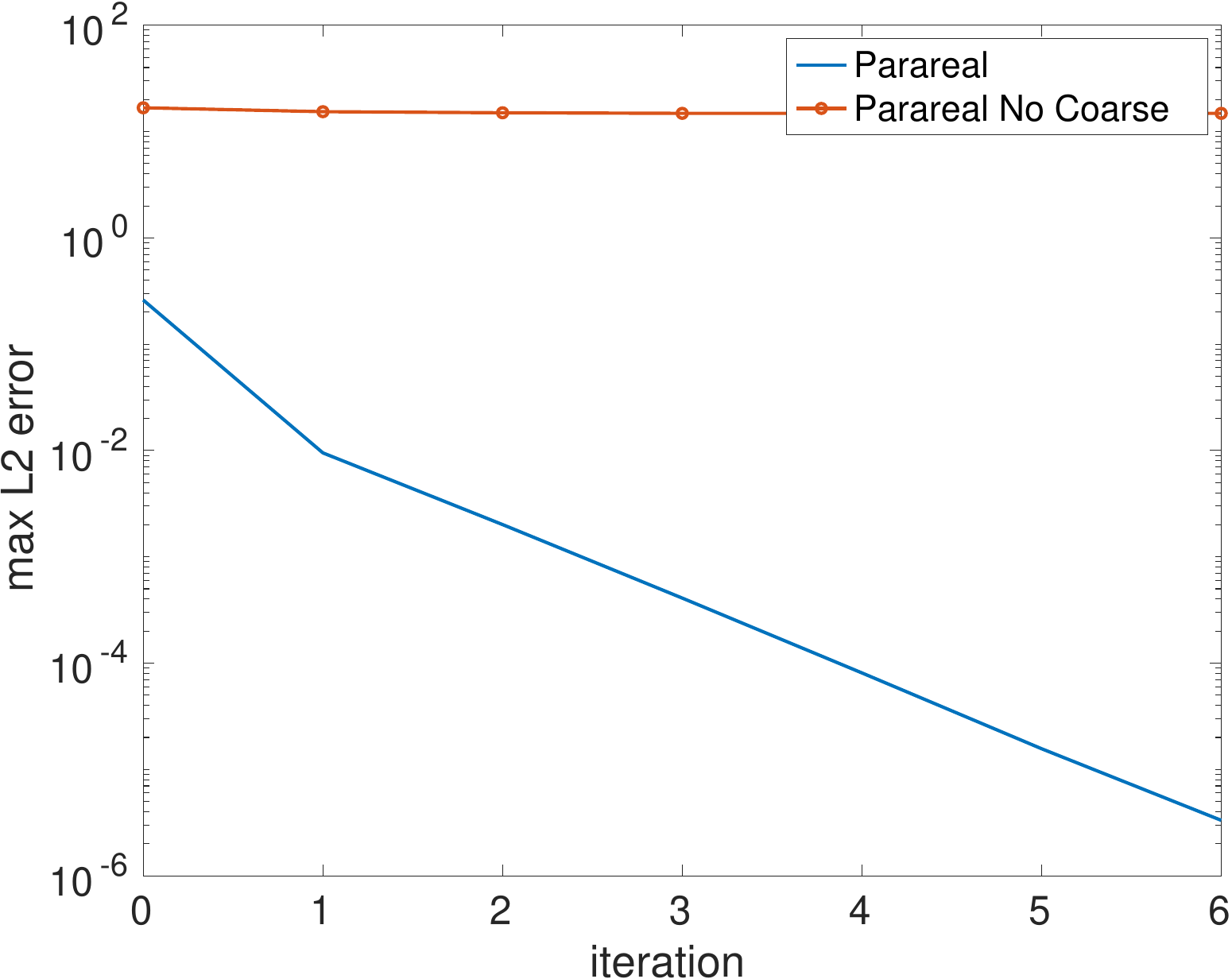}\qquad
    \includegraphics[width=0.42\textwidth,height=0.31\textwidth]{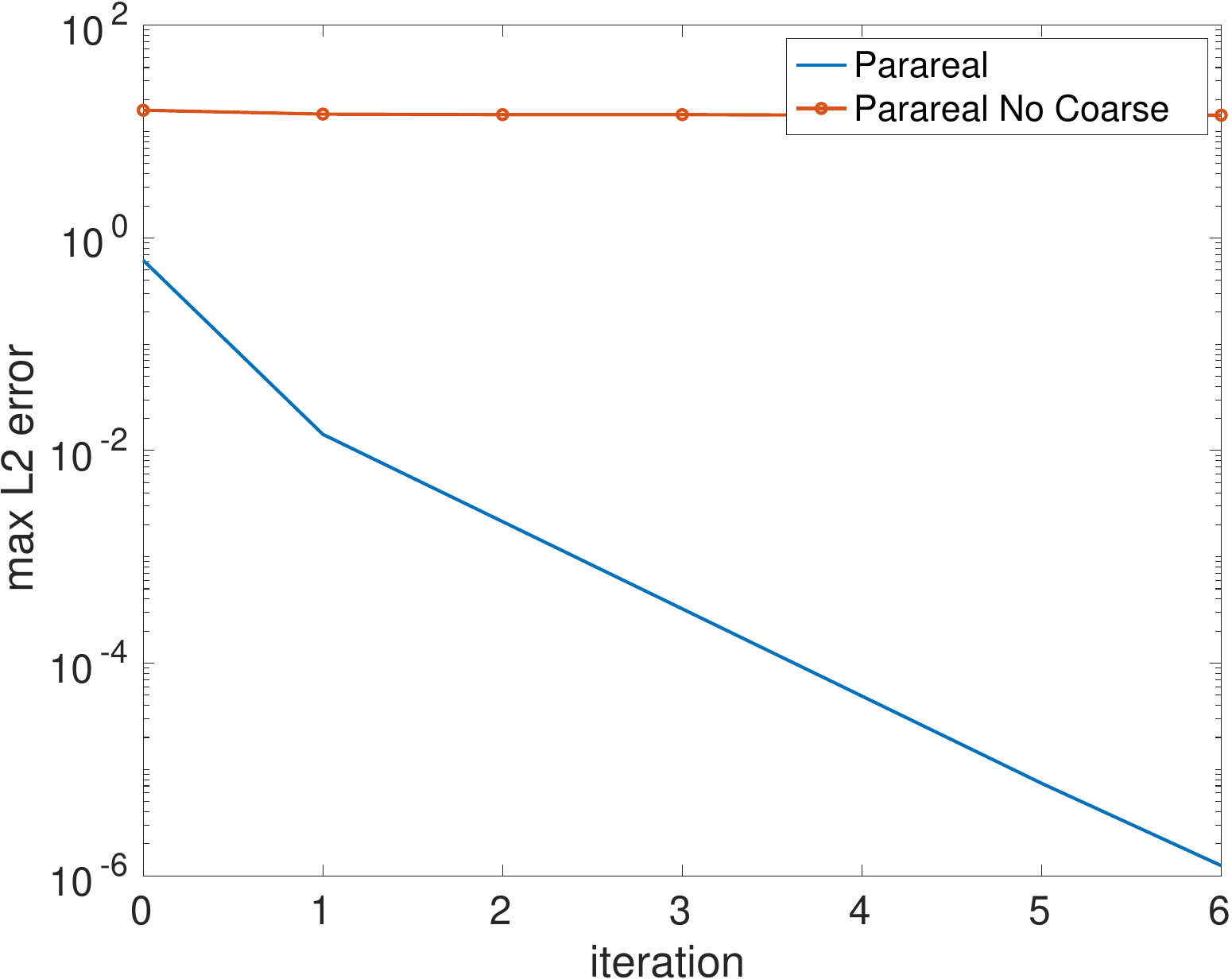}}
  \mbox{\includegraphics[width=0.42\textwidth,height=0.31\textwidth]{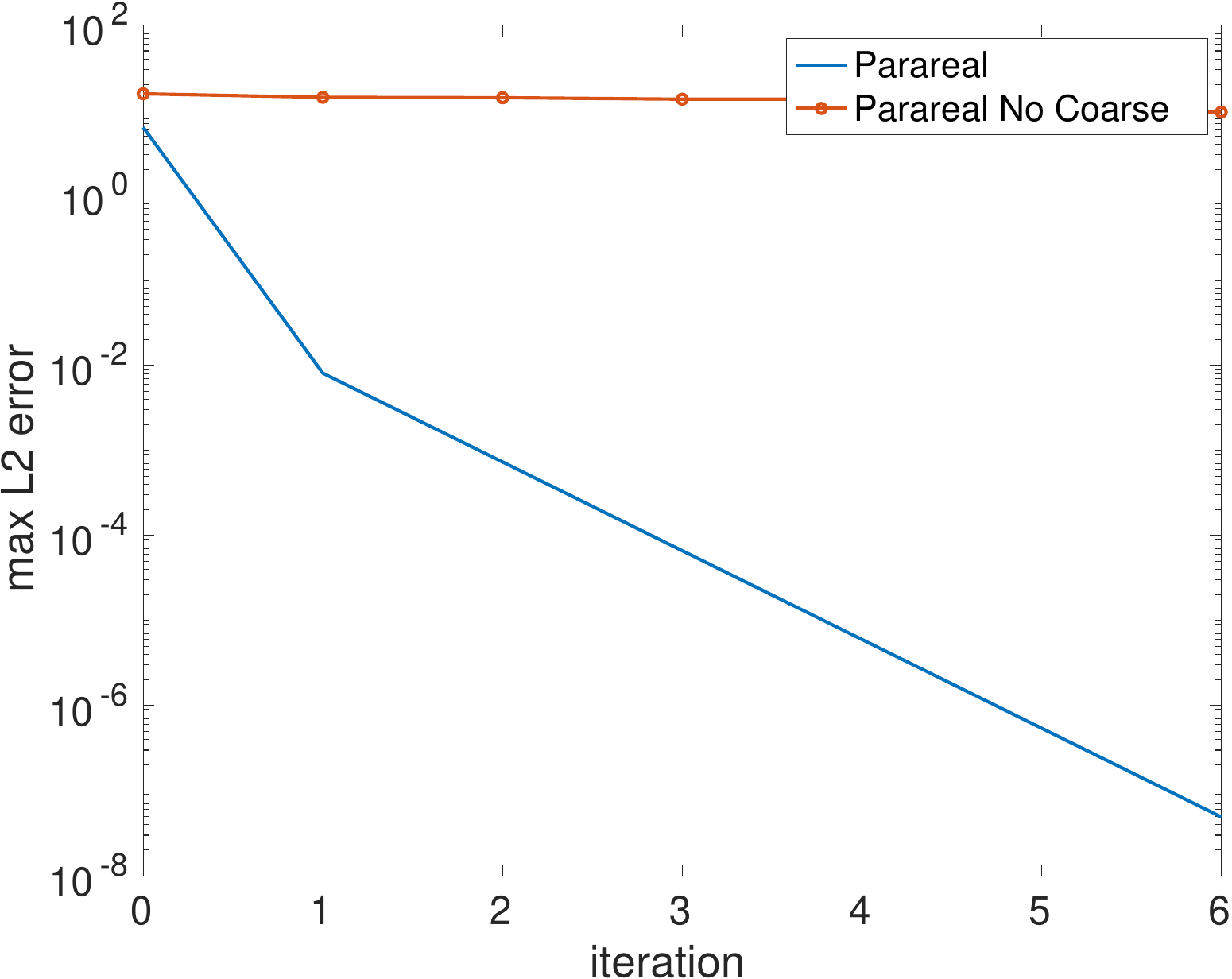}\qquad
    \includegraphics[width=0.42\textwidth,height=0.31\textwidth]{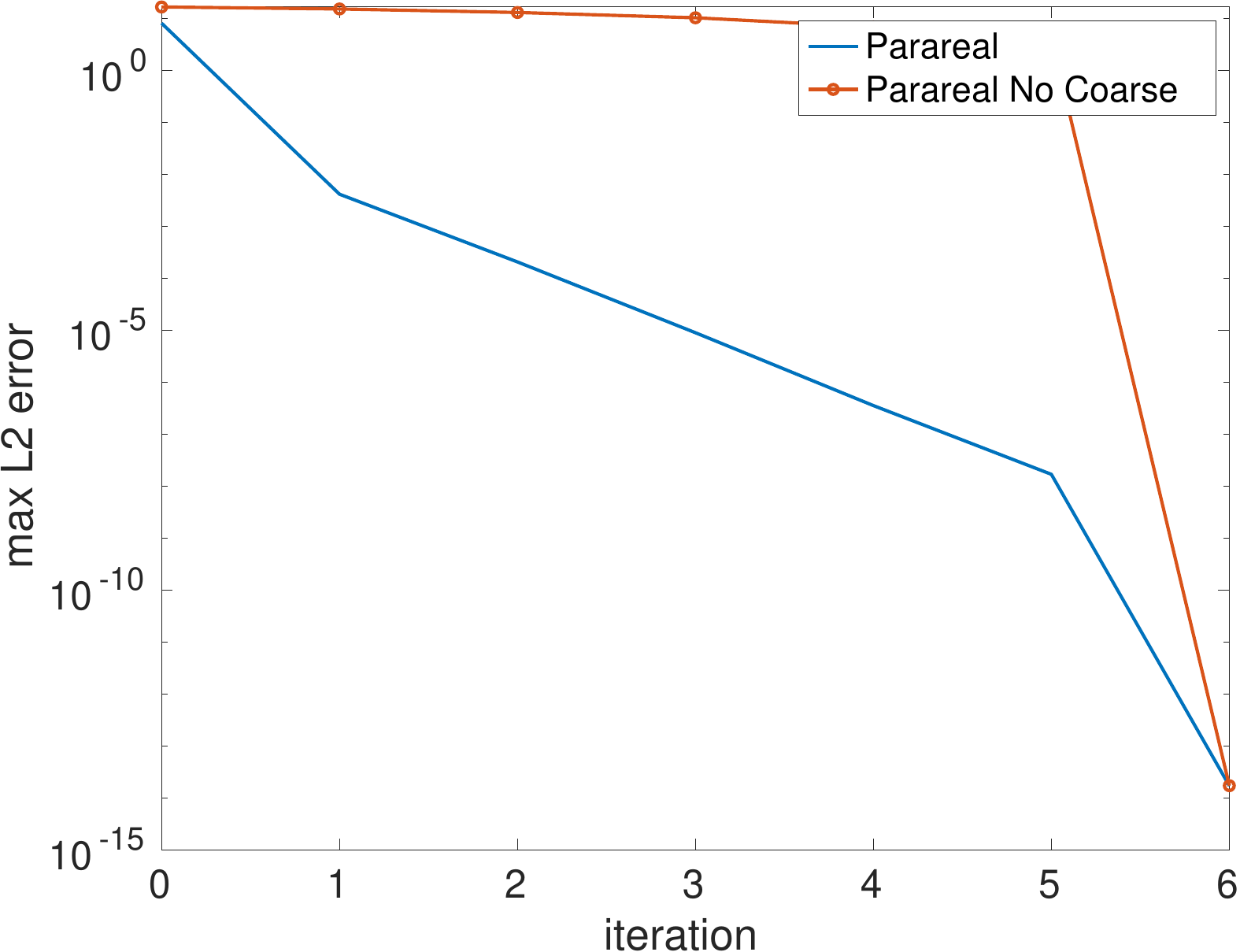}}
  \caption{Classical Parareal with and without coarse propagator with
    zero Neumann boundary conditions.  From top left to bottom right:
    48, 24, 12, and 6 coarse time intervals and $T=3$.}
  \label{PararealFigWithoutCoarseNeumann}
\end{figure}
We see that while the Parareal algorithm with coarse propagator is
converging as before with Dirichlet conditions, the Parareal algorithm
without coarse propagator is not converging any more, except in the
last case when only 6 coarse time intervals are used, and here it is
the finite step convergence property of the Parareal algorithm when
iterating 6 times that leads to convergence, there is no contraction
without the coarse propagator with Neumann boundary conditions.

It is very helpful at this point to take a closer look at the
solutions we were computing in these two experiments for zero
Dirichlet and zero Neumann conditions, shown in Figure \ref{SolFigs}
in the two leftmost panels.
\begin{figure}[t]
  \centering
  \mbox{\includegraphics[width=0.19\textwidth,trim=300 30 310
      30,clip]{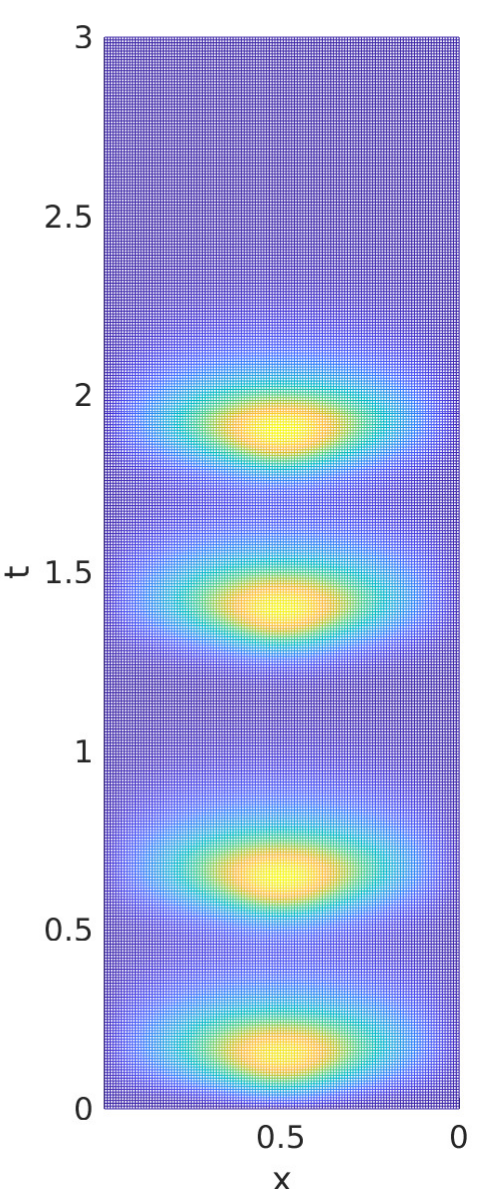}
   \includegraphics[width=0.19\textwidth,trim=300 30 310
     30,clip]{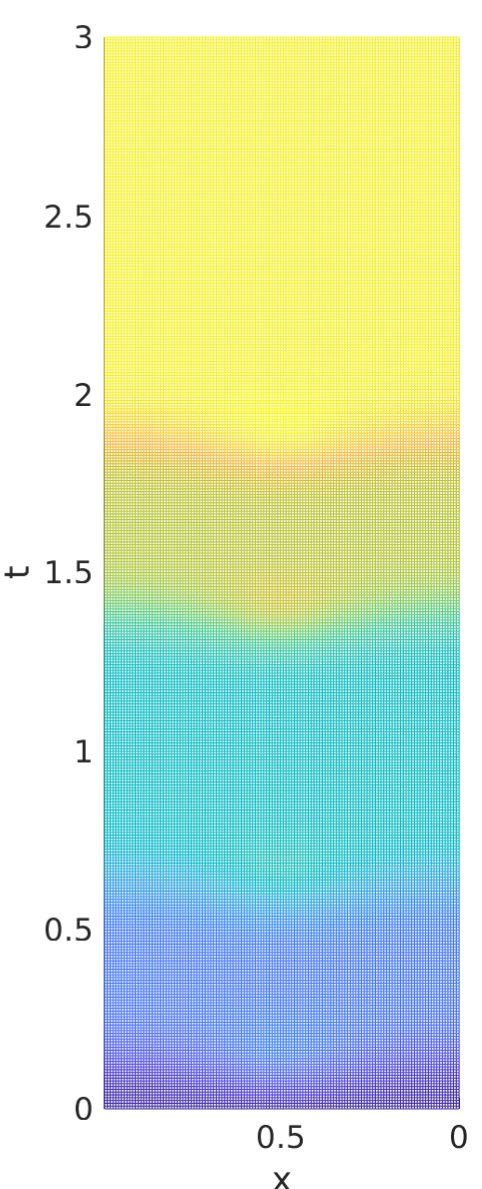}
   \includegraphics[width=0.19\textwidth,trim=300 30 310
        30,clip]{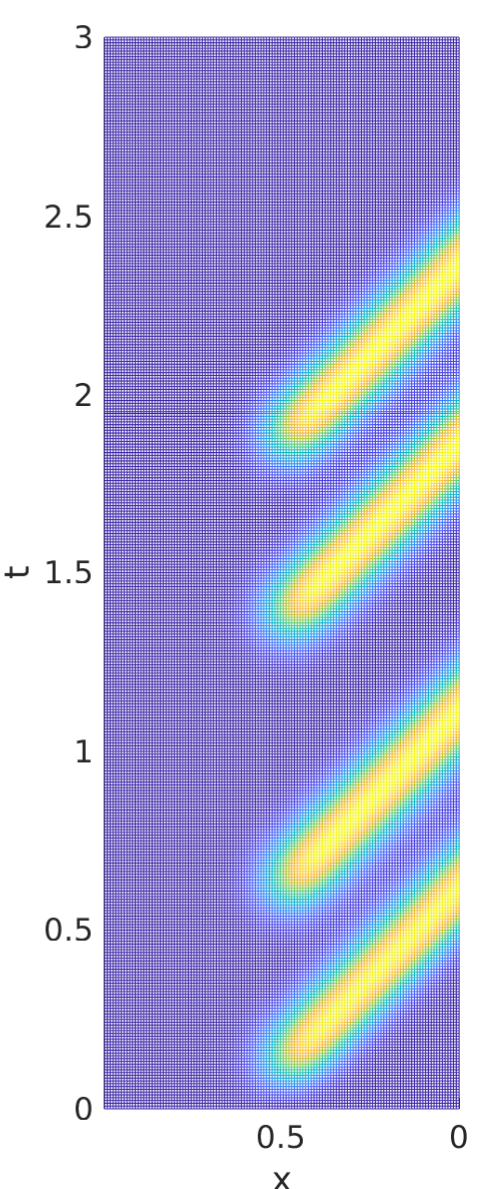}      
   \includegraphics[width=0.19\textwidth,trim=300 30 310
     30,clip]{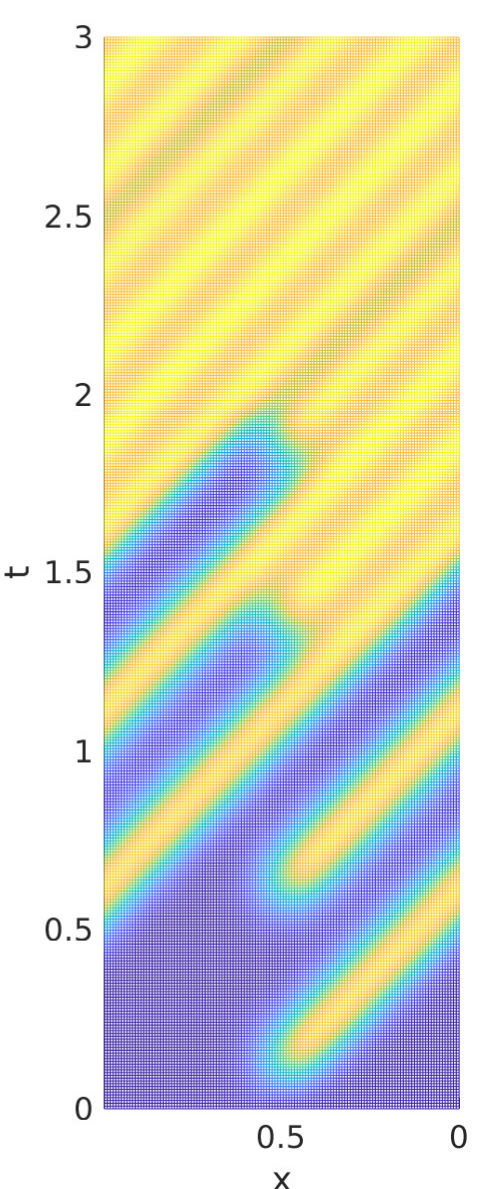}
   \includegraphics[width=0.19\textwidth,trim=300 30 310
     30,clip]{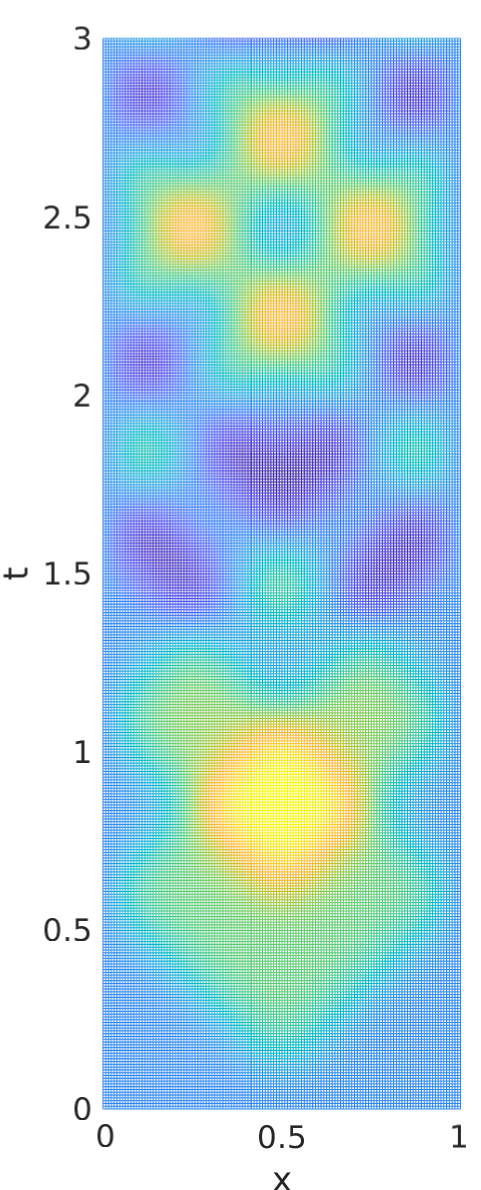}}
  \caption{From left to right: solution of the heat equation model
    problem with zero Dirichlet and zero Neumann boundary
    conditions. Solution of the corresponding advection equation with
    Dirichlet boundary condition on the left and periodic boundary
    condition, and solution of the second order wave equation with
    zero Dirichlet boundary conditions.}
  \label{SolFigs}
\end{figure}
We see that the solution with Dirichlet boundary conditions is very
much local in time: in order to compute the solution around $t=2$, we
only need to know if the right hand side function was on just before,
but not at the much earlier times $t<<2$, since the solution does not
contain this information any more. This is very different in the case
of Neumann conditions, when using the same source term: now the
solution at time $t=2$ strongly depends on the fact that the right
hand side function was on at earlier times, since the heat injected into
the system is kept, as a constant value in space. In analogy, it is
very easy to predict the temperature in the room you are currently in
a year from now, you only need to know if the heater or air-conditioner
is on and the windows and doors are open or closed then, but not
what happened with the room over the entire year before.

\section{Convergence Analysis}

In order to study our observation mathematically, it is easiest to
consider a Parareal algorithm which uses a spectral method in space
and is performing exact integrations in time for our heat equation
model problem \eqref{HeatEquation1D}. The fine propagator
$F(U_n^k,T_n,T_{n+1})$ is thus using a highly accurate spectral method
in space with $m_F$ basis functions to solve \eqref{HeatEquation1D},
and the coarse propagator $G(U_n^k,T_n,T_{n+1})$ solves the same
problem using a very cheap spectral method in space using only
$m_G<<m_F$ basis functions. Note that we allow explicitly that
$m_G=0$, which means the coarse solver is not present in the Parareal
algorithm \eqref{Parareal}! 

A spectral representation of the solution of the heat equation
\eqref{HeatEquation1D} can be obtained using separation of variables,
and considering a spatial domain $\Omega:=(0,\pi)$ to avoid to have to
carry the constant $\pi$ in all the computations (see also Remark
\ref{Remark3} later), we get
\begin{equation}
  u(x,t):=\sum_{m=1}^{\infty}\hat{u}_{m}(t)\sin(m x).
\end{equation}
Expanding also the right hand side in a sine series,
\begin{equation}
  f(x,t):=\sum_{m=1}^{\infty}\hat{f}_{m}(t)\sin(m x),
\end{equation}
and the initial condition,
\begin{equation}
  u_0(x):=\sum_{m=1}^{\infty}\hat{u}_{0,m}(t)\sin(m x),
\end{equation}
we find that the Fourier coefficients in the solution satisfy
\begin{equation}
  \partial_t\hat{u}_{m}=-m^2\hat{u}_{m}+\hat{f}_{m},\quad
    \hat{u}_{m}(0)=\hat{u}_{0,m}.
\end{equation}
The solution of this equation can be written in closed form using an
integrating factor,
\begin{equation}
  \hat{u}_{m}(t)=\hat{u}_{0,m}e^{-m^2t}
    +\int_0^t\hat{f}_{m}(\tau)e^{-m^2(t-\tau)}d\tau.
\end{equation}
If we use in the Parareal algorithm \eqref{Parareal} a spectral
approximation of the solution for the fine propagator using $m_F$
modes, we obtain 
\begin{equation}\label{FineApproximation}
  F(U_n^k,T_n,T_{n+1})=\sum_{m=1}^{m_F}\left(\hat{U}_{m,n}^k
  e^{-m^2\Delta T}+\int_{T_n}^{T_{n+1}}\hat{f}_{m}(\tau)
  e^{-m^2(t-\tau)}d\tau\right)\sin(m x),
\end{equation}
and similarly for the coarse approximation when using $m_G$ modes
\begin{equation}\label{CoarseApproximation}
  G(U_n^k,T_n,T_{n+1})=\sum_{m=1}^{m_G}\left(\hat{U}_{m,n}^k
  e^{-m^2\Delta T}+\int_{T_n}^{T_{n+1}}\hat{f}_{m}(\tau)
  e^{-m^2(t-\tau)}d\tau\right)\sin(m x).
\end{equation}
Because of the orthogonality of the sine functions, the Parareal
algorithm \eqref{Parareal} is diagonalized in this representation,
which is the main reason that we chose a spectral method, to simplify
the present analysis for this short note. For Fourier modes with $m\le
m_G$, the first and the third term in \eqref{Parareal} coincide, as we
see from \eqref{FineApproximation} and \eqref{CoarseApproximation},
and thus cancel, and hence the Parareal update formula simplifies to
\begin{equation}\label{UpdateXiSmall}
    \begin{multlined}[c][0.85\displaywidth]
        \hat{U}_{m,n+1}^{k+1}=\hat{U}_{m,n}^{k+1}
        e^{-m^2\Delta T}+\int_{T_n}^{T_{n+1}}\hat{f}_{m}(\tau)
        e^{-m^2(t-\tau)}d\tau, \\ ~\qquad\mbox{for $m=1,2,\ldots,m_G$.}
    \end{multlined}
\end{equation}
For $m>m_G$, the coarse correction term is not present in the Parareal update
formula \eqref{Parareal}, and thus only the contribution from the fine propagator remains,
\begin{equation}\label{UpdateXiLarge}
    \begin{multlined}[c][0.85\displaywidth]
        \hat{U}_{m,n+1}^{k+1}\!=\!\hat{U}_{m,n}^{k}
        e^{-m^2\Delta T}\!+\!\int_{T_n}^{T_{n+1}}\!\!\!\hat{f}_{m}(\tau)
        e^{-m^2(t-\tau)}d\tau, \\ \mbox{for $m=m_G+1,m_G+2,\ldots,m_F$.}
    \end{multlined}
\end{equation}
The only difference is the iteration index $k$ in the first term on
the right, but this term makes all the difference: for $m\le m_G$,
the update formula \eqref{UpdateXiSmall} represent the exact
integration of the corresponding mode sequentially going through the
entire time domain $(0,T)$. For $m> m_G$, the update formula
\eqref{UpdateXiLarge} represents simply a block Jacobi iteration
solving all time subintervals $(T_n,T_{n+1})$ in parallel starting
from the current approximation at iteration $k$. Using this insight,
we obtain the following convergence result.

\begin{theorem}[Parareal convergenve even without coarse propagator]\label{TH1}
  The Parareal algorithm \eqref{Parareal} for the heat equation
  \eqref{HeatEquation1D} using the spectral coarse propagator
  \eqref{CoarseApproximation} and the spectral fine propagator
  \eqref{FineApproximation} satisfies for any initial guess $U_n^0$
  the convergence estimate
  \begin{equation}\label{ConvergenceEstimate}
    \sup_n||U_n^k(\cdot)-u(\cdot,T_n)||_2\le
    e^{-(m_g+1)^2k\Delta T}\sup_n||U_n^0(\cdot)-u(\cdot,T_n)||_2,
  \end{equation}
  and this estimate also holds if the coarse propagator does not
  contain any modes, $m_G=0$, which means it is not present.
  The Parareal algorithm therefore converges also without coarse
  propagator.
\end{theorem}
\begin{proof} We introduce the error in Fourier space at the
interfaces, $\hat{E}_{m,n}^k:=\hat{U}_{m,n}-\hat{U}_{m,n}^k$,
where the converged solution $\hat{U}_{m,n}$ satisfies
\begin{equation}\label{ExactSolFormula}
  \hat{U}_{m,n+1}=\hat{U}_{m,n}
  e^{-m^2\Delta T}+\int_{T_n}^{T_{n+1}}\hat{f}_{m}(\tau)
  e^{-m^2(t-\tau)}d\tau, \quad \mbox{for all $m=1,2,\ldots$.}
\end{equation}
Taking the difference with the Parareal update formula for small $m$
in \eqref{UpdateXiSmall} and using the fact that
$U_{m,0}^k=U_{m,0}=u^0$, we thus obtain a vanishing error in the
first Fourier coefficients after having performed one Parareal
iteration,
$$
  \hat{E}_{m,n}^k=0,\quad m=1,2,\ldots,m_G,\quad k=1,2,\ldots.
$$
For $m>m_G$, taking the difference between \eqref{ExactSolFormula}
and the update formula for larger $m$ in \eqref{UpdateXiLarge}, the
part with the source term $f$ cancels, and we are left with
$$
  \hat{E}_{m,n}^{k+1}=\hat{E}_{m,n}^ke^{-m^2\Delta T},\quad
   m=m_G+1,m_G+2,\ldots.
$$
Now using the Parseval-Plancherel identity, and taking the largest
convergence factor $e^{-(m_G+1)^2\Delta T}$ out of the sum, we obtain
the convergence estimate \eqref{ConvergenceEstimate} taking the sup
over all time intervals $n$.
\end{proof}
We show in Figure \ref{FigComparison} a graphical comparison of the
convergence factor of each Fourier mode without coarse propagator from
\eqref{ConvergenceEstimate}, i.e. $e^{-m^2\Delta T}$, and with coarse
propagator $\frac{|e^{-m^2\Delta T}-R_G(-m^2\Delta
  T)|}{1-|R_G(-m^2\Delta T)|}$ from \cite{gander2007analysis} for one
Backward Euler step, $R_G(z)=\frac{1}{1-z}$, as we used for Figures
\ref{PararealFigWithoutCoarse} and
\ref{PararealFigWithoutCoarseNeumann}.
\begin{figure}[t]
  \centering
  \mbox{
    \begin{overpic}[width=0.45\textwidth,trim=90 260 100 130,clip]{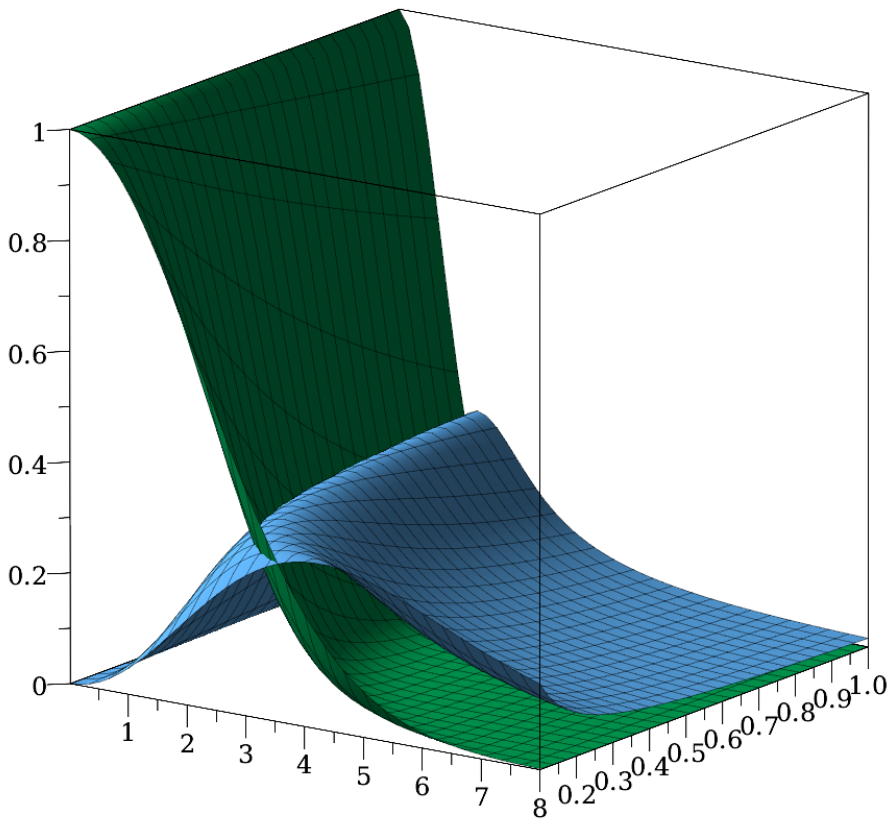}
        \put(32.8,2.5){$m$}
        \put(74.3,1.7){$\Delta T$}
    \end{overpic}
    \fbox{\begin{overpic}[width=0.5\textwidth,trim=62 363 94 55,clip]{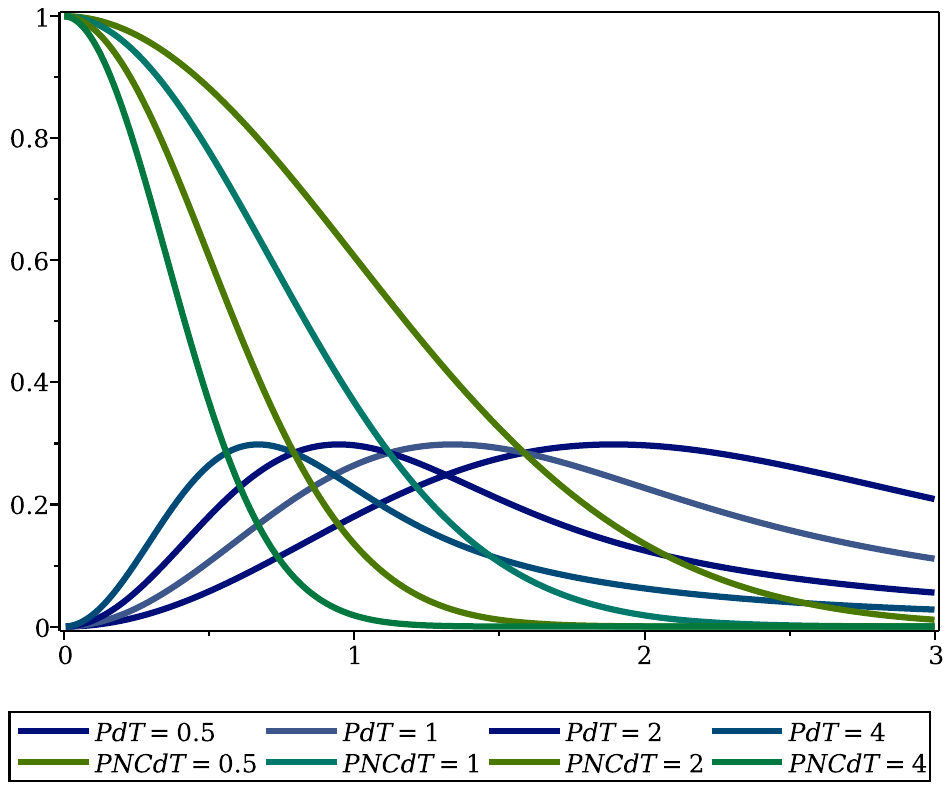}
        \put(51.5,8.5){\small $m$}
    \end{overpic}}
  }
  \caption{Contraction for each Fourier mode $m$ as function of the
    Parareal time interval length $\Delta T$ with (blue, PdT=$\Delta
    T$) and without (green, PNCdT=$\Delta T$) coarse propagator.}
  \label{FigComparison}
\end{figure}
We see that indeed there are situations where the coarse correction
with one Backward Euler step is detrimental, as observed in Figures
\ref{PararealFigWithoutCoarse} and
\ref{PararealFigWithoutCoarseNeumann}, and there is also quantitative
agreement.

\begin{remark}
  Theorem \ref{TH1} shows that the Parareal algorithm \eqref{Parareal}
  for the heat equation is scalable, provided the time interval length
  $\Delta T$ is held constant: no matter how many such time intervals
  are solved simultaneously, the convergence rate remains the same,
  and this even in the case without coarse propagator! This is like
  the scalability of one level Schwarz methods discovered in
  \cite{cances2013domain} for molecular simulations in solvent models,
  and rigorously proved in
  \cite{ciaramella2017analysis,ciaramella2018analysis,ciaramella2018analysis2}
  using three different techniques: Fourier analysis, maximum
  principle, and projection arguments in Hilbert spaces.
\end{remark}

\begin{remark}
  If we had zero Neumann boundary conditions, a similar result could
  also be derived using a cosine expansion, but then no coarse
  correction would not lead to a convergent method, since the zero
  mode would have the contraction factor one. One must therefore in
  this case have at least the constant mode in the coarse correction
  for the Parareal algorithm to converge. Otherwise however all
  remains the same.
\end{remark}

\begin{remark}\label{Remark3}
  A precise dependence on the domain size can be obtained by
  introducing the spatial domain of length $L$, $\Omega=(0,L)$. In
  this case the convergence factor simply becomes
  $e^{-\left(\frac{(m_G+1)\pi}{L}\right)^2\Delta T}$. The same
  estimate can also be obtained for a more general parabolic equation,
  the convergence factor then becomes $e^{-\lambda_{m_G+1}\Delta T}$,
  where $\lambda_{m_G+1}$ is the $m_G$ plus first eigenvalue of the
  corresponding spatial operator.
\end{remark}

\section{Conclusions}

We have shown that Parareal without a coarse propagator can converge
very well for the prototype parabolic problem of the heat equation,
because such problems ``forget'' all fine information over time, and
only very coarse information, if at all, remains, depending on the
boundary conditions. This is very different for hyperbolic problems:
while for the advection equation $u_t+au_x=f$ with Dirichlet boundary
condition, whose solution is shown in Figure \ref{SolFigs} (middle)
the information is also ``forgotten'' by leaving the domain on the
right and Parareal converges very well \cite{gander2008analysis}, also
without coarse propagator, for the advection equation with periodic
boundary conditions in the fourth plot of Figure \ref{SolFigs} this is
not the case any more: nothing is forgotten, and the coarse propagator
needs to have fine accuracy quality for the Parareal algorithm not to
fail. The situation is similar for the second order wave equation
$u_{tt}=u_{xx}$, whose solution is shown on the right in Figure
\ref{SolFigs}.\bigskip

\paragraph{Acknowledgement.} Funded by the Deutsche Forschungsgemeinschaft
(DFG, German Research Foundation) under Germany's Excellence Strategy
EXC 2044 -- 390685587, Mathematics Münster: Dynamics–Geometry–Structure
and the Swiss National Science Foundation.

\bibliographystyle{plain}
\bibliography{paper.bib}

\begin{thebibliography}{10}

\bibitem{cances2013domain}
Eric Cances, Yvon Maday, and Benjamin Stamm.
\newblock Domain decomposition for implicit solvation models.
\newblock {\em The Journal of chemical physics}, 139(5), 2013.

\bibitem{ciaramella2017analysis}
Gabriele Ciaramella and Martin~J. Gander.
\newblock Analysis of the parallel {S}chwarz method for growing chains of
  fixed-sized subdomains: {P}art {I}.
\newblock {\em SIAM Journal on Numerical Analysis}, 55(3):1330--1356, 2017.

\bibitem{ciaramella2018analysis}
Gabriele Ciaramella and Martin~J. Gander.
\newblock Analysis of the parallel {S}chwarz method for growing chains of
  fixed-sized subdomains: {P}art {II}.
\newblock {\em SIAM Journal on Numerical Analysis}, 56(3):1498--1524, 2018.

\bibitem{ciaramella2018analysis2}
Gabriele Ciaramella and Martin~J. Gander.
\newblock Analysis of the parallel {S}chwarz method for growing chains of
  fixed-sized subdomains: {P}art {III}.
\newblock {\em Electron. Trans. Numer. Anal}, 49:201--243, 2018.

\bibitem{gander2008analysis}
Martin~J. Gander.
\newblock Analysis of the parareal algorithm applied to hyperbolic problems
  using characteristics.
\newblock {\em SeMA Journal: Bolet{\'\i}n de la Sociedad Espa{\~n}ola de
  Matem{\'a}tica Aplicada}, (42):21--36, 2008.

\bibitem{gander2008nonlinear}
Martin~J. Gander and Ernst Hairer.
\newblock Nonlinear convergence analysis for the parareal algorithm.
\newblock In {\em Domain decomposition methods in science and engineering
  XVII}, pages 45--56. Springer, 2008.

\bibitem{gander2014analysis}
Martin~J. Gander and Ernst Hairer.
\newblock Analysis for parareal algorithms applied to hamiltonian differential
  equations.
\newblock {\em Journal of Computational and Applied Mathematics}, 259:2--13,
  2014.

\bibitem{gander2007analysis}
Martin~J. Gander and Stefan Vandewalle.
\newblock Analysis of the parareal time-parallel time-integration method.
\newblock {\em SIAM Journal on Scientific Computing}, 29(2):556--578, 2007.

\bibitem{lions2001resolution}
Jacques-Louis Lions, Yvon Maday, and Gabriel Turinici.
\newblock R{\'e}solution d'edp par un sch{\'e}ma en temps parar{\'e}el.
\newblock {\em Comptes Rendus de l'Acad{\'e}mie des Sciences-Series
  I-Mathematics}, 332(7):661--668, 2001.

\bibitem{takami2014identity}
Toshiya Takami and Daiki Fukudome.
\newblock An identity parareal method for temporal parallel computations.
\newblock In {\em Parallel Processing and Applied Mathematics: 10th
  International Conference, PPAM 2013, Warsaw, Poland, September 8-11, 2013,
  Revised Selected Papers, Part I 10}, pages 67--75. Springer, 2014.

\end{thebibliography}

\end{document}